\documentclass[11pt]{amsart}

\usepackage{amsthm}
\usepackage{amsmath}
\usepackage{amssymb}

\newtheorem{thm}{Theorem}[section]
\newtheorem{theorem}[thm]{Theorem}

\newtheorem{corollary}[thm]{Corollary}

\newtheorem{lemma}[thm]{Lemma}
\newtheorem{proposition}[thm]{Proposition}

\newtheorem{definition}[thm]{Definition}

\theoremstyle{remark}
\newtheorem{remark}[thm]{Remark}

\newcommand{\norm}[1]{\|#1\|}

\newcommand{\RR}{\mathbb R}

\newcommand{\CC}{\mathbb C}
\newcommand{\HH}{\mathbb H}

\usepackage{enumerate}

\DeclareMathAlphabet{\mathpzc}{OT1}{pzc}{m}{it}

\begin{document}

\title{Weak Phase Retrieval and Phaseless Reconstruction}
\author[Botelho-Andrade, Casazza, Ghoreishi, Jose, Tremain
 ]{Sara Botelho-Andrade, Peter G. Casazza,
 Dorsa Ghoreishi, Shani Jose, Janet C. Tremain}
\address{Department of Mathematics, University
of Missouri, Columbia, MO 65211-4100}

\thanks{The authors were supported by
 NSF DMS 1609760; NSF ATD 1321779; and ARO  W911NF-16-1-0008}

\email{sandrade102087@gmail.com, ‎Casazzap@missouri.edu; }
\email{dorsa.ghoreishi@gmail.com, shanijose@gmail.com}
\email{‎  ‎  Tremainjc@missouri.edu}


\begin{abstract} Phase retrieval and phaseless reconstruction for Hilbert space frames is a very active area of research.  Recently,
it was shown that these concepts are equivalent.  In this
paper, we make a detailed study of a weakening of these
concepts to weak phase retrieval and weak phaseless
reconstruction.  We will give several necessary and/or
sufficient conditions for frames to have these weak
properties.  We will prove three surprising results:
(1)  Weak phaseless reconstruction is equivalent to
phaseless reconstruction.  I.e.  It never was {\it weak};
(2)  Weak phase retrieval is not equivalent to weak
phaseless reconstruction; (3)  Weak phase retrieval requires at least $2m-2$ vectors in an m-dimensional
Hilbert space.  We also gives several examples illustrating
the relationship between these concepts.
\end{abstract}

\maketitle
\section{Introduction}
The problem of retrieving the phase of a signal, given a set of intensity measurements, has been studied by engineers for many years. Signals passing through linear systems often result in lost or distorted phase information. This partial loss of phase information occurs in various applications including speech recognition~\cite{BeRi99,RaJu93,ReBlScCa004}, and optics applications such as X-ray crystallography~\cite{BaMn86,Fi78,Fi82}. The concept of \textit{phase retrieval} for Hilbert space
frames was introduced in 2006 by Balan, Casazza, and Edidin~\cite{BCE} and since then it has become an active area of research. Phase retrieval deals with recovering the phase of a signal given intensity measurements from a redundant linear system. In phaseless reconstruction the unknown signal itself is reconstructed from these measurements. In recent literature, the two terms were used interchangeably.  However it is not obvious from the definitions that the two are equivalent. Recently, authors in~\cite{SaCa016} proved that phase retrieval is equivalent to phaseless reconstruction in both the real and complex case.

Phase retrieval has been defined for vectors as well as for projections. \textit{Phase retrieval by projections} occur in real life problems, such as crystal twinning~\cite{Dr010}, where the signal is projected onto some higher dimensional subspaces and has to be recovered from the norms of the projections of the vectors onto the subspaces. We refer the reader to~\cite{CCPW} for a detailed study of phase retrieval by projections. At times these projections are identified with their target spaces. Determining when subspaces $\{W_i\}_{i=1}^n$ and $\{W_i^\perp\}_{i=1}^n$ both do phase retrieval has given way to the notion of \textit{norm retrieval} \cite{BaCaCaJaWo014}, another important area of research.

While investigating the relationship between phase retrieval and phaseless reconstruction, in ~\cite{SaCa016} it was noted that if two vectors have the same phase then they will be zero in the same coordinates. This gave way to a weakening of phase retrieval, known as weak phase retrieval. In this work, we study the weakened notions of phase retrieval and phaseless reconstruction. One limitation of current methods used for retrieving the phase of a signal is computing power. Recall that a generic family of $(2m-1)$-vectors in $\mathbb{R}^m$ does phaseless reconstruction, however no set of $(2m -2)$-vectors can (See \cite{BCE} for details). By generic we are referring to an open dense set in the set of $(2m -1)$-element frames in $\HH^m$.
  We started with the motivation that weak phase retrieval could be done with $m+1$ vectors in $\mathbb{R}^m$. However, it will be shown that the cardinality condition can only be relaxed to $2m-2$. Nevertheless, the results we obtain in this work are interesting in their own right and contribute to the overall understanding of phase retrieval. We provide illustrative examples in the real and complex cases for weak phase retrieval.

The rest of the paper is organized as follows: In Section~\ref{s:prilim}, we give basic definitions and certain preliminary results to be used in the paper. Weak phase retrieval is defined in Section~\ref{s:weakphaseretrieval}. Characterizations are given in both real and complex case. Also, the minimum number of vectors needed for weak phase retrieval is obtained. In Section~\ref{s:weakphaseless}, we define weak phaseless reconstruction and prove that it is equivalent to phase retrieval in the real case. We conclude by providing certain illustrative examples in Section~\ref{s:examples}.

\section{Preliminaries}\label{s:prilim}

In this section, we introduce some of the basic definitions and results from frame theory. Throughout this paper, $\mathbb{H}^m$ denotes an $m$ dimensional real or complex Hilbert space and we will write $\mathbb{R}^m$ or $\mathbb{C}^m$ when it is necessary to differentiate between the two. We start with the definition of a frame in $\HH^m$.

\begin{definition}\label{D:frame}
A family of vectors $\Phi=\{\phi_i\}_{i=1}^n$ in $\HH^m$ is a {\bf frame} if there are constants $0<A\leq B<\infty$ so that for all $x\in \HH^m$
\[A \|x\|^2 \leq \sum_{i=1}^n |\langle x, \phi_i\rangle|^2\leq B\norm{x}^2, \]
where $A$ and $B$ are the {\bf lower and upper frame bounds} of the frame, respectively. The frame is called an {\bf A-tight frame} if $A=B$ and is a {\bf Parseval frame} if $A=B=1$.
\end{definition}

In addition, $\Phi$ is called an {\bf equal norm frame} if $\norm{\phi_i}=\norm{\phi_j}$ for all $i, j$ and is called a {\bf unit norm frame} if $\norm{\phi_i}=1$ for all $i=1, 2, \ldots n$.

Next, we give the formal definitions of phase retrieval, phaseless reconstruction, and norm retrieval. Note that, here, phase of vector $x=re^{it}$ is taken as $e^{it}$. 

\begin{definition}\label{D:phase_ret&phaseless}
Let $\Phi=\{\phi_i\}_{i=1}^n \in \mathbb{H}^m$ be such that for $x, y\in \mathbb{H}^m$
\[ |\langle x,\phi_i\rangle|=|\langle y,\phi_i\rangle|, \mbox{ for all }i=1,2,\ldots,n. \]
$\Phi$ yields
 \begin{itemize}
  \item[(i)]~ \textbf{phase retrieval} with respect to an
  orthonormal basis $\{e_i\}_{i=1}^m$ if there is a $|\theta|=1$ such that $x$ and $\theta y$ have the same phase.  I.e.  $x_i=\theta y_i$, for all $i=1,2,\ldots,m$, where $x_i = \langle x,e_i\rangle$.

  \item[(ii)]~ \textbf{phaseless reconstruction} if there is a $|\theta|=1$ such that $x=\theta y$.
  \item[(iii)]~ \textbf{norm retrieval} if $\|x\|=\|y\|$.
 \end{itemize}
\end{definition}

We note that tight frames $\{\phi_i\}_{i=1}^m$ for
$\HH^m$ do norm retrieval.  Indeed, if
\[ |\langle x,\phi_i\rangle|=|\langle y,\phi_i\rangle|,
\mbox{ for all }i=1,2,\ldots,m,\]
then
\[ A\|x\|^2 = \sum_{i=1}^m|\langle x,\phi_i\rangle|^2
= \sum_{i=1}^m|\langle y,\phi_i\rangle|^2 = A \|y\|^2.\]
Phase retrieval in $\RR^m$ is classified in terms of a fundamental result called the complement property, which we define below:

\begin{definition}[\cite{BCE}]\label{D:complement_prop}
A frame $\Phi=\{\phi_i\}_{i=1}^n $in $\HH^m$ satisfies the {\bf complement property} if for all subsets ${I}\subset\{1, 2, \ldots, n\}$, either span$\{\phi_i\}_{i\in I}=\HH^m$ or span$\{\phi_i\}_{i\in I^c} = \HH^m$.
\end{definition}

A fundamental result from \cite{BCE} is:

\begin{theorem}[\cite{BCE}]
If $\Phi$ does phaseless reconstruction then it has complement property.  In $\RR^m$, if $\Phi$ has complement property then it does phase retrieval.
\end{theorem}

It follows that if $\Phi=\{\phi_i\}_{i=1}^n$ does phase retrieval in $\RR^m$ then $n\ge 2m-1$.

Full spark is another important notion of vectors in frame theory. A formal definition is given below:

\begin{definition}\label{D:full_Spark}
Given a family of vectors $\Phi=\{\phi_i\}_{i=1}^n$ in $\HH^m$, the {\bf spark} of $\Phi$ is defined as the cardinality of the smallest linearly dependent subset of $\Phi$. When spark$(\Phi) = m + 1$, every subset of size $m$ is linearly independent, and in that case, $\Phi$ is said to be {\bf full spark}.
\end{definition}

We note that from the definitions it follows that full spark frames with $n\ge 2m-1$ have the complement property and hence do phaseless reconstruction.  Moreover, if $n=2m-1$ then the complement property clearly implies full spark.  



\section{Weak Phase Retrieval}\label{s:weakphaseretrieval}

In this section, we define the notion of weak phase retrieval and make a detailed study of it. We obtain the minimum number of vectors required to do weak phase retrieval. First we define the notion of vectors having weakly the same phase. 

\begin{definition}
Two vectors in $\HH^m$, $x=(a_1,a_2,\ldots,a_m)$ and $y=(b_1,b_2,\ldots,b_m)$ {\bf weakly have the same phase} if there is a $|\theta|=1$ so that
\[ \mbox{phase}(a_i)= \theta \mbox{phase}(b_i), \mbox{ for all }i=1,2,\ldots,m, \mbox{ for which } a_i\not= 0 \not= b_i.\]
In the real case, if $\theta=1$ we say $x,y$ {\bf weakly have the same signs} and if $\theta=-1$ they {\bf weakly have opposite signs.}
\end{definition}

 In the definition above note that we are only comparing the phase of $x$ and $y$ for entries where both are nonzero. Hence, two vectors may \textit{weakly} have the same phase but not have the same phase in the usual sense.  We define weak phase retrieval formally as follows:

\begin{definition}
A family of vectors $\{\phi_i\}_{i=1}^n$ in $\HH^m$ does {\bf weak phase retrieval} if for any $x=(a_1,a_2,\ldots,a_m)$ and $y=(b_1,b_2,\ldots,b_m)$ in $\HH^m$, with
\[ |\langle x,\phi_i\rangle|=|\langle y,\phi_i\rangle|,\mbox{ for all }i=1,2,\ldots,m,\]
then $x,y$ weakly have the same phase.
\end{definition}

Observe that the difference between phase retrieval and weak phase retrieval is that in the later it is possible for $a_i =0$ but $b_i \not= 0$.

\subsection{Real Case}

Now we begin our study of weak phase retrieval in $\RR^m$. The following proposition provides a useful criteria for determining when two vectors have weakly the same or opposite phases. In what follows, we use $[m]$ to denote the set $\{1, 2, \ldots, m\}$.

\begin{proposition}\label{prop1}
Let $x=(a_1,a_2,\ldots,a_m)$ and $y=(b_1,b_2,\ldots,b_m)$ in $\RR^m$.  The following are equivalent:
\begin{enumerate}
\item We have
\[sgn\ (a_ia_j)=sgn\ (b_ib_j),\mbox{ for all }a_ia_j\neq 0\neq b_ib_j.\]
\item Either $x,y$ have weakly the same signs or they have weakly opposite signs.
\end{enumerate}
\end{proposition}

\begin{proof}
$(1) \Rightarrow (2)$:
Let
\[ I = \{1\le i \le m:a_i=0\}\mbox{ and } J=\{1\le i \le n:b_i=0\}.\]
Let
\[ K=[m]\setminus (I \cup J).\]
So $i\in K$ if and only if $a_i \not= 0 \not= b_i$. Let $i_0 = min\ K$.  We examine two cases:
\vskip12pt
\noindent {\bf Case 1}:  $sgn\ a_{i_0} = sgn\ b_{i_0}$.
\vskip12pt
For any $i_0 \not= k \in K$, $sgn\ (a_{i_0}a_k)= sgn\ (b_{i_0}b_k)$, implies $sgn\ a_k = sgn\ b_k$.  Since all other coordinates of either $x$ or $y$ are zero, it follows that $x,y$ weakly have the same signs.

\vskip12pt
\noindent {\bf Case 2}:  $sgn\ a_{i_0}=-sgn\ b_{i_0}$.
\vskip12pt
For any $i_o \not= k \in K$, $a_{i_0}a_k = b_{i_0}b_k$ implies $sgn\ a_k = - sgn\ b_k$.  Again, since all other coordinates of either $x$ or $y$ are zero, it follows that $x,y$ weakly have opposite signs.
\vskip12pt
$(2)\Rightarrow (1)$:  This is immediate.
\end{proof}

The next lemma will be useful in the following proofs as it gives a criteria for showing when vectors do not weakly
have the same phase.

\begin{lemma}\label{lem1}
Let $x=(a_1,a_2,...,a_m)$ and $y=(b_1,b_2,...,b_m)$ be vectors in $\RR^m$. If there exists $i\in [m]$ such that $a_ib_i\neq 0$ and $\langle x,y\rangle =0$, then $x$ and $y$ do not have weakly the same or opposite signs.
\end{lemma}
\begin{proof}
We proceed by way of contradiction.
If $x$ and $y$ weakly have the same phase then $a_jb_j\geq 0 $ for all $j\in [m]$ and in particular we arrive at the following contradiction
\[\langle x,y\rangle =\sum_{j=1}^n a_jb_j\geq a_ib_i>0.\]
If $x$ and $y$ weakly have opposite phases then $a_jb_j\leq 0$ for all $j\in [m]$ and by reversing the inequalities in the expression above we get the desired result.
\end{proof}

The following result relates weak phase retrieval and phase retrieval. Recall that in the real case, it is known that phase retrieval, phaseless reconstruction and the complement property are equivalent~\cite{BCE, SaCa016}.

\begin{corollary}\label{C:disjointsupp}
Suppose $\Phi=\{\phi_i\}_{i=1}^n\in \RR^m$ does weak phase retrieval but fails complement property.  Then there exists two vectors $v,w \in \RR^m$ such that $v \perp w$ and
\begin{equation}\label{E:innerpdt-codition}
 |\langle v,\phi_i\rangle|=|\langle w,\phi_i\rangle| \mbox{ for all i}.
\end{equation}
Further, $v$ and $w$ are disjointly supported.
\end{corollary}

\begin{proof} By the assumption, $\Phi=\{\phi_i\}_{i=1}^n$ fails complement property so there exists  $I\subset [n]$, s.t. $A= Span\{\phi_i\}_{i\in I}$ $\neq$ $\RR^m$ and $B= Span\{\phi_i\}_{i\in I^c}$ $\neq$ $\RR^m$. Choose $\norm{x}=\norm{y}=1$ such that $x\perp A$ and $y\perp B$. Then
\begin{equation}\nonumber
 |\langle x+y,\phi_i\rangle|=|\langle x-y,\phi_i\rangle|
\mbox{ for all i=1, 2, \ldots, n}.
\end{equation}
Let $w=x+y$ and $v=x-y$. Then $v\perp w$. Observe
\[\langle w,v\rangle=\langle x+y,x-y\rangle = \|x\|^2+\langle y,x\rangle -\langle x,y\rangle-\|y\|^2=0.\]
Moreover, the assumption that $\Phi$ does weak phase retrieval implies $v$ and $w$ have weakly the same or opposite phases. Then it follows from Lemma \ref{lem1} that $v_iw_i=0$
for all $i=1,2,\ldots,m$ and so $v$ and $w$ are disjointly supported.
\end{proof}

\noindent {\bf Example}:  In $\RR^2$ let $\phi_1=(1,1)$ and
and $\phi_2=(1,-1)$.  These vectors clearly fail complement
property.  But if $x=(a_1,a_2)$, $b=(b_1,b_2)$ and we have,
\[ |\langle x,\phi_i\rangle|=|\langle y,\phi_i\rangle|,
\mbox{ for }i=1,2,\]
then
\[ |a_1+a_2|^2=|b_1+b_2|^2 \mbox{ and } |a_1-a_2|^2
=|b_1-b_2|^2.\]
By squaring these out and subtracting the result we get:
\[ 4a_1a_2=4b_1b_2.\]
Hence, either $x,y$ have the same signs or opposite signs.
I.e.  These vectors do weak phase retrieval.
\vskip12pt

With some particular assumptions, the following proposition gives the specific form of vectors which do weak phase retrieval but not phase retrieval.

\begin{proposition}
Let $\Phi=\{\phi_i\}_{i=1}^n\in \RR^m$ be such that $\Phi$ does weak phase retrieval but fails complement property. Let $x=(a_1, a_2, \ldots, a_m),$ $y=(b_1, b_2, \ldots b_m) \in \RR^m $ such that $x+y\perp x-y$ and satisfy equation~\eqref{E:innerpdt-codition}. If $a_ib_i\neq 0$, $a_jb_j\neq 0$ for some $i, j$ and all other co-ordinates of $x$ and $y$ are zero, then
\[|a_i|=|b_i|,\mbox{ for }i=1,2.\]
\end{proposition}

\begin{proof}
Without loss of generality, take $x= (a_1,a_2,0,\ldots,0)$ and $y=(b_1,b_2, \\ 0,\ldots,0)$. Observe that both $x+y$ and $x-y$ either weakly have the same phase or weakly have the opposite phase. Thus, by Corollary~\ref{C:disjointsupp}, $x+y$ and $x-y$ have disjoint support as these vectors are orthogonal.
Also,
\[ x+y=(a_1+b_1,a_2+b_2,0,\ldots,0)\mbox{ and }
x-y=(a_1-b_1,a_2-b_2,0,\ldots,0).\]
Since $x+y$ and $x-y$ are disjointly supported, it reduces to the cases where either $a_1=\pm b_1$ and $
 a_2=\pm b_2$. In both cases, it follows from equation~\eqref{E:innerpdt-codition} that $|a_i|=|b_i|$
 for all $i=1,2,\ldots,m.$
\end{proof}

The next theorem gives the main result about the minimum number of vectors required to do weak phase retrieval in $\RR^m$.  Recall that phase retrieval requires $n\ge 2m-1$
vectors.

\begin{theorem}\label{T:WPhaseRetbound}
If $\{\phi_i\}_{i=1}^n$ does weak phase retrieval on $\RR^m$ then $n\geq 2m-2$.
\end{theorem}

\begin{proof}
For a contradiction assume $n\leq 2m-3$ and choose $I\subset [n]$ with $I=[m-2]$.  Then $|I|=m-2$ and $|I^c|\leq m-1$. For this partition of $[n]$, let $x+y$ and $x-y$ be as in the proof of Corollary \ref{C:disjointsupp}. Then $x+y$ and $x-y$ must be disjointly supported which follows from the Corollary~\ref{C:disjointsupp}. Therefore for each $i=1,2,\dots,m$, $a_i=\epsilon_i b_i$, where $\epsilon_i=\pm 1$ for each $i$ and $a_i, b_i$ are the ith coordinates of $x$ and $y$, respectively. Observe the conclusion holds for a fixed $x$ and any $y\in (\text{span} \{\phi\}_{i\in I})^\perp$ and $\dim \;(\text{span} \{\phi_i\}_{i\in I})^\perp \geq 2$. However this poses a contradiction since there are infinitely many distinct choices of $y$ in this space, while our argument shows that there are at most $2^m$ possibilities for $y$.
\end{proof}

Contrary to the initial hopes, the previous result shows that the minimal number of vectors doing weak phase retrieval is only one less than the number of
vectors doing phase retrieval. However it is interesting to note that a minimal set of vectors doing weak phase retrieval is necessarily full spark, as is true for
the minimal number of vectors doing phase retrieval,
as the next result shows.

\begin{theorem}\label{T:fullspark}
If $\Phi=\{\phi_i\}_{i=1}^{2n-2}$ does weak phase retrieval in $\RR^n$, then $\Phi$ is full spark.
\end{theorem}

\begin{proof}
 We proceed by way of contradiction.  Assume $\Phi$ is not full spark.  Then there exists $I\subset \{1,2,...,2n-2\}$ with $|I|=n$ such that $\dim \text{span} \{\phi_i\}_{i\in I}\leq n-1$.
Observe that the choice of $I$ above implies $|I^c|=n-2$. Now we arrive at a contradiction by applying the same argument used in (the proof of) Theorem \ref{T:WPhaseRetbound}.
\end{proof}

It is important to note that the converse of Theorem \ref{T:fullspark} does not hold. For example, the canonical basis in $\RR^2$ is trivially full spark but does not do weak phase retrieval.

If $\Phi$ is as in Theorem \ref{T:fullspark}, then the following corollary guarantees it is possible to add a vector to this set and obtain a collection which does phaseless reconstruction.

\begin{corollary}
If $\Phi$ is as in Theorem \ref{T:fullspark}, then there exists a dense set of vectors $\mathsf{F}$ in $\RR^n$ such that $\{\psi\}\cup\Phi$ does phaseless reconstruction for any  $\psi\in\mathsf{F}$.
\end{corollary}

\begin{proof}
We observe that the set of $\psi\in\RR^n$ such that $\Phi\cup \{\psi\}$ is full spark is dense in $\RR^n$. To see this let $\mathsf{G}=\bigcup_{\substack{I\subset [2n-2]\\ |I|=n-1}} \text{span} \{\phi_i\}_{i\in I}$. Then $\mathsf{G}$ is the finite union of hyperplanes so $\mathsf{G}^c$ is dense and $\{\psi\}\cup\Phi$ is full spark for any $\psi\in\mathsf{G}^c$. To verify that this collection of
vectors is full spark. Note that either a sub-collection of m-vectors is contained in $\Phi$, then it spans $\RR^n$, or the subcollection contains the vector $\psi$. In this case, denote $I\subset [2n-2]$ with $|I|=n-1$ and suppose $\sum_{i\in I} a_i\phi_i+a\psi =0$. Therefore $a\psi=-\sum_{i\in I}a_i\phi_i$ and if $a\neq 0$ then $a\psi\in\text{span} \{\phi_i\}_{i\in I}$, a contradiction. It follows $a=0$ and since $\Phi$ is full spark (see Theorem \ref{T:fullspark}), in particular $\{\phi_i\}_{i\in I}$ are linearly independent, it follows that $a_i=0$ for all $i\in I$.
\end{proof}

\subsection{Complex Case}
An extension of Proposition~\ref{prop1} in the complex case is given below:

\begin{proposition}\label{prop1-complex}
Let $x=(a_1, a_2, \ldots, a_m)$ and $y=(b_1, b_2, \ldots, b_m)$ in $\CC^m$. The following are equivalent:
\begin{enumerate}
  \item If there is a $|\theta|=1$ such that $ phase\ (a_i)=\theta phase\ (b_i)$, for some $i$, then
    $phase\ (a_ia_j)={\theta}^2 phase\ (b_ib_j), \ i\not= j \ \text{and} \  a_j\not= 0 \not= b_j.$

  \item $x$ and $y$ weakly have the same phase.
\end{enumerate}
\end{proposition}

\begin{proof}
$(1) \Rightarrow (2)$: Let the index sets $I, J$ and $K$ be as in proposition~\ref{prop1}. By $(1)$, there is a $|\theta|=1$ such that $phase\ (a_i)=\theta phase\ (b_i)$ for some $i\in K$.

Now, for any $j\in K,\ j\not= i$, \[phase\ (a_ia_j)=phase\ (a_i)\ phase\ (a_j)=\theta phase\ (b_i)\ phase\ (a_j).\] But $phase\ (a_ia_j)={\theta}^2\ phase\ (b_ib_j)={\theta}^2\ phase\ (b_i) phase\ (b_j)$. Thus, it follows that $phase\ (a_j)=\theta\ phase\ (b_j)$. Since all other coordinates of either $x$ or $y$ are zero, it follows that $x,y$ weakly have the same phase.
\vskip12pt
$(2) \Rightarrow (1)$: By definition, there is a $|\theta |=1$ such that $phase\ (a_i)=\theta\ phase\ (b_i)$ for all $a_i\not=0\not=b_i.$ Now, $(1)$ follows immediately as $phase\ (a_ia_j)=phase\ (a_i) phase\ (a_j).$
\end{proof}

\section{Weak Phaseless Reconstruction}\label{s:weakphaseless}

In this section, we define weak phaseless reconstruction and study its characterizations. A formal definition is given below:

\begin{definition}
A family of vectors $\{\phi_i\}_{i=1}^n$ in $\HH^m$ does {\bf weak phaseless reconstruction} if for any $x=(a_1,a_2,\ldots,a_m)$ and $y=(b_1,b_2,\ldots,b_m)$
in $\HH^m$, with
\begin{equation}\label{E:innerpdtcondition_defn}
|\langle x,\phi_i\rangle|=|\langle y,\phi_i\rangle|,
\mbox{ for all }i=1,2,\ldots,m,
\end{equation}
there is a $|\theta|=1$ so that
\[ a_i= \theta b_i, \mbox{ for all }i=1,2,\ldots,m, \mbox{ for which } a_i\not= 0 \not= b_i.\]
In particular, $\{\phi_i\}$ does phaseless reconstruction for vectors having all non-zero coordinates.
\end{definition}

Note that if $\Phi = \{\phi_i\}_{i=1}^n \in \RR^m$ does weak phaseless reconstruction, then it does weak phase retrieval. The converse is not true in general. Let $x=(a_1,a_1,...,a_m)$ and $y=(b_1,b_2,...,b_m)$. If $\Phi=\{\phi_i\}_{i=1}^n\in \RR^m$ does weak phase retrieval and $|\{i|a_ib_i\neq 0\}|=2$ then $\Phi$ may not do weak phaseless reconstruction. If $a_ib_i=a_jb_j$  where $a_ib_i\neq 0$ and $a_jb_j\neq 0$ then we certainly cannot conclude in general that $|a_i|=|b_i|$ (see Example~\ref{Ex:R3}).

\begin{theorem}\label{thm1}
If $x=(a_1,a_2,\ldots,a_m)$ and $y=(b_1,b_2,\ldots,b_m)$ in $\RR^m$.  The following are equivalent:

\begin{enumerate}
\item[I.] There is a $\theta = \pm 1$ so that
\[ a_i = \theta b_i,\mbox{ for all } a_i \not= 0 \not= b_i.\]

\item[II.] We have  \,$a_ia_j=b_ib_j$ for all $1\le i, j \le m$, and $|a_i|=|b_i|$ for all $i$ such that $a_i\not=0 \not= b_i.$

\item[III.] The following hold:
\begin{enumerate}
\item[A.] Either $x,y$ have weakly the same signs or they have weakly the opposite
signs.

\item[B.] One of the following holds:
 \begin{itemize}
    \item[(i)]  There is a $1\le i \le m$ so that $a_i=0$ and $b_j=0$ for all $j\not= i$.

    \item[(ii)]  There is a $1\le i \le m$ so that $b_i=0$ and $a_j=0$ for all $j\not= i$.

    \item[(iii)]  If (i) and (ii) fail and $I = \{1\le i \le m:a_i\not= 0 \not= b_i\}$, then the following hold:
     \begin{itemize}
        \item[(a)]  If $i\in I^c$ then $a_i=0$ or $b_i=0$.

        \item[(b)]  For all $i\in I$, $|a_i|=|b_i|$.
     \end{itemize}
 \end{itemize}
\end{enumerate}
\end{enumerate}
\end{theorem}

\begin{proof}
$(I) \Rightarrow (II):$
By $(I)$ $a_i=\theta b_i$ for all $i$ such that both are non-zero, so $a_ia_j=(\theta b_i)(\theta b_j)$ and so $a_ia_j=\theta^2 b_ib_j$. Since $\theta=\pm 1$ it follows that $a_ia_j=b_ib_j$ for all $i,j$ (that are non-zero). The second part is trivial.
\vspace*{.5 cm}\\
$(II)\Rightarrow (III):$

(A)  This follows from Proposition \ref{prop1}.
\vskip12pt
(B)  (i)  Assume $a_i=0$ but $b_i\not= 0$.  Then for all $j\not= i$ we have
$a_ia_j=0=b_ib_j$ and so $b_j=0$.

(ii)  This is symmetric to (i).

(iii)  If (i) and (ii) fail, then by definition, for any $i$, either both $a_i$ and $b_i$ are zero or they are both non-zero, which proves (A). (B) is immediate.
\vskip10pt

$(III)\Rightarrow (I):$
The existence of $\theta$ is clear by part $A$. In part $B$, $(i)$ and $(ii)$ trivially imply ($I$). Assume $(iii)$ then for each $i$ such that $a_i\neq 0\neq b_i$ and $|a_i|=|b_i|$ then $a_i=\pm b_i$.
\end{proof}

\begin{corollary}
Let $\Phi$ be a frame for $\RR^m$.  The following are equivalent:
\begin{enumerate}
\item $\Phi$ does weak phaseless reconstruction.
\item For any $x=(a_1,a_2,\ldots,a_m)$ and $y=(b_1,b_2,\ldots,b_m)$
in $\RR^m$, if
\[ |\langle x,\phi_i\rangle|=|\langle y,\phi_i\rangle|
\mbox{ for all i},\]
then each of the equivalent conditions in Theorem~\ref{thm1} holds.
\end{enumerate}
\end{corollary}

The following theorems provide conditions under which weak phase retrieval is equivalent to weak phaseless reconstruction.

\begin{proposition}
Let $\Phi=\{\phi_i\}_{i=1}^n$ do weak phase retrieval on vectors $x=(a_1,a_2,...,a_m)$ and $y=(b_1,b_2,...,b_m)$ in $\HH^m$. If $|I|=|\{i:a_ib_i\neq 0\}|\geq 3$ and $a_ia_j=b_ib_j$ for all $i,j\in I$, then $\Phi$ does weak phaseless reconstruction.
\end{proposition}

\begin{proof}
If $i,j,k$ are three members of $I$ with $a_ia_j=b_ib_j$, $a_ia_k=b_ib_k$ and $a_ka_j=b_kb_j$, then a short calculation gives $a_i^2a_ja_k=b_i^2b_jb_k$ and hence $|a_i|=|b_i|$. This computation holds for each $i\in I$ and
since $\Phi$ does phase retrieval,
there is a $|\theta|=1$ so that $phase\ a_i=\theta \ phase
\ b_i$ for all $i$.  It follows that $a_i=\theta \ b_i$ for
all $i=1,2,\ldots,m$.
\end{proof}

It turns out that whenever a frame contains the unit
vector basis, then weak phase retrieval and phaseless
reconstruction are the same.

\begin{proposition}\label{T:unitvector_case}
Let the frame $\Phi =\{\phi_i\}_{i=1}^n \in \RR^m$ does weak phase retrieval. If $\Phi$ contains the standard basis vectors, then $\Phi$ does phaseless reconstruction.
\end{proposition}

\begin{proof}
Let $x=(a_1, a_2, \ldots, a_m)$, $y=(b_1, b_2, \ldots, b_m) \in \RR^m$. By definition of weak phase retrieval, $\Phi$ satisfies the equation~\ref{E:innerpdtcondition_defn}. In particular, for $\phi_i=e_i$, the equation~\ref{E:innerpdtcondition_defn} implies that $|a_i|=|b_i|,~\forall i=1, 2, \ldots, m$. Hence the theorem.
\end{proof}

We conclude this section by showing the surprising
result that weak phaseless reconstruction is same as phaseless reconstruction in $\RR^m$, i.e. it is not really {\bf weak}.
\begin{theorem}
Frames which do weak phaseless reconstruction
in $\RR^m$ do phaseless reconstruction.
\end{theorem}

\begin{proof}
For a contradiction assume $\Phi=\{\phi\}_{i=1}^n\subset \RR^m$ does weak phaseless reconstruction but fails the complement property. Then there exists $I\subset [n]$ such that $\text{Span}_{i\in I}\;\phi_i\neq \RR^m$ and $\text{Span}_{i\in I^C}\;\phi_i\neq \RR^m$. Pick non-zero vectors $x,y\in\RR^m$ such that $x\perp \text{Span}_{i\in I}\;\phi_i\neq \RR^m$ and $y\perp \text{Span}_{i\in I^C}\;\phi_i\neq \RR^M$. Then for any $c\neq 0$ we have
\[|\langle x+cy,\phi_i\rangle |=|\langle x-cy,\phi_i\rangle |\;\;\;\text{ for all } i\in[n].\]

Now we consider the following cases where $x_i$ and $y_i$ denotes the $i^\text{th}$ coordinate of the vectors $x$ and $y$.
\begin{enumerate}
\item[Case 1:] $\{i:x_i\neq 0\}\cap \{i:y_i\neq 0\}=\emptyset$\\
Set $c=1$ and observe since $x\neq 0$ there exists some $i\in [n]$ such that $x_i\neq 0$ and $y_i=0$ and similarly there exists $j\in [n]$ such that $y_j\neq 0$ but $x_j=0$. Then $x+y$ and $x-y$ have the same sign in the $i^\text{th}$-coordinate but opposite signs in the $j^\text{th}$ coordinate, this contradicts the assumption that $\Phi$ does weak phaseless reconstruction.

\item[Case 2:] There exists $i, j\subset [n]$ such that $x_iy_i\neq 0$ and $x_j=0$, $y_j\neq 0$.\\
Without loss of generality, we may assume $x_iy_i>0$ otherwise consider $-x$ or $-y$. If $0< c\leq \frac{x_i}{y_i}$, then the $i^\text{th}$ coordinate of $x+cy$ and $x-cy$ have the same sign whereas the $j^\text{th}$ coordinates have opposite signs which contradicts the assumption. By considering $y+cx$ and $y-cx$ this argument holds in the case that $y_j=0$ and $x_j\neq 0$.

\item[Case 3:] $x_i=0$ if and only if $y_i=0$.\\
By choosing $c$ small enough, we have that $x_i+cy_i\not=0$
if and only if $x_i-cy_i\not= 0$.  By weak phase retrieval,
there is a $|d|=1$ so that $x_i+cy_i=d(x_i-cy_i)$.
But this forces either $x_i\not= 0$ or $y_i\not= 0$ but
not both which contradicts the assumption for case 3.
\end{enumerate}
\end{proof}

It is known \cite{BaCaCaJaWo014} that if $\Phi=\{\phi_i\}_{i=1}^n$ does phase retrieval or phaseless reconstruction
in $\HH^m$ and $T$ is an invertible operator on $\HH^m$
then $\{T\phi_i\}_{i=1}^n$ does phase retrieval.  It now
follows that the same result holds for weak phaseless
reconstruction.  However, this result does not hold for
weak phase retrieval.  Indeed, if $\phi_1=(1,1)$ and
$\phi_2=(1,-1)$, then we have seen that this frame does
weak phase retrieval in $\RR^2$.  But the invertible
operator $T(\phi_1)=(1,0),\ T(\phi_2)=(0,1)$ maps this
frame to a frame which fails weak phase retrieval.

\section{Illustrative Examples}\label{s:examples}  

In this section, we provide examples of frames that do weak phase retrieval in $\RR^3$ and $\RR^4$. As seen earlier, the vectors $(1, 1)$ and $(1, -1)$ do weak phase retrieval in $\RR^2$ but fail phase retrieval.

Our first example is a frame which does weak phase retrieval but fails weak phaseless reconstruction.
\noindent {\bf Example}:\label{Ex:R3} We work with the row vectors of

\[\Phi=\begin{bmatrix}
\phi_1&\vline&1&1&1\\
\phi_2&\vline&-1&1&1\\
\phi_3&\vline&1&-1&1\\
\phi_4&\vline&1&1&-1
\end{bmatrix}\]

Observe that the rows of this matrix form an equal norm tight frame $\Phi$ (and hence do norm retrieval).
If $x=(a_1,a_2,a_3)$ the following is the coefficient matrix where the row $E_i$ represents the coefficients obtained from the expansion $|\langle x,\phi_i\rangle |^2$

\[1/2\begin{bmatrix}
&\vline&a_1a_2&a_1a_3&a_2a_3&\sum_{i=1}^3 a_i^2\\
E_1&\vline&1&1&1&1/2\\
E_2&\vline&-1&-1&1&1/2\\
E_3&\vline&-1&1&-1&1/2\\
E_4&\vline&1&-1&-1&1/2
\end{bmatrix}\]
Then the following row operations give

\[1/2\begin{bmatrix}
&\vline&a_1a_2&a_1a_3&a_2a_3&\sum_{i=1}^3 a_i^2\\
F_1=E_1-E_2&\vline&1&1&0&0\\
F_2=E_3-E_4&\vline&-1&1&0&0\\
F_3=E_1-E_3&\vline&1&0&1&0\\
F_4=E_2-E_4&\vline&-1&0&1&0\\
F_4=E_1-E_4&\vline&0&1&1&0\\
F_5=E_2-E_3&\vline&0&-1&1&0
\end{bmatrix}\]

\[1/2\begin{bmatrix}
&\vline&a_1a_2&a_1a_3&a_2a_3&\sum_{i=1}^3 a_i^2\\
F_1-F_2&\vline&1&0&0&0\\
F_3+F_4&\vline&0&0&1&0\\
F_5-F_6&\vline&0&1&0&0
\end{bmatrix}\]
Therefore we have demonstrated a procedure to identify $a_ia_j$ for all $1\le i \not= j \le 3$. This shows that given $y=(b_1,b_2,b_3)$ satisfying $|\langle x,\phi_i\rangle|^2=|\langle y,\phi_i\rangle|^2$ then by the procedure outlined above we obtain
\[ a_ia_j = b_ib_j,\mbox{ for all } 1 \le i \not= j \le 3.\]
By Proposition \ref{prop1}, these four vectors do weak sign retrieval in $\RR^3$. However this family fails to do weak phaseless reconstruction. Observe the vectors $x=(1,2,0)$ and $y=(2,1,0)$ satisfy $|\langle  x,\phi_i\rangle |=|\langle  y,\phi_i\rangle |$ however do not have the
same absolute value in each coordinate.
\vskip12pt

Our next example is a frame which does weak phase retrieval
 but fails phaseless reconstruction.
\noindent {\bf Example}:\label{Ex:R4}
We provide a set of six vectors in $\RR^4$ which does weak phase retrieval in $\RR^4$. In this case our vectors are the rows of the matrix:

\[\Phi=\begin{bmatrix}
\phi_1&\vline&1&1&1&-1\\
\phi_2&\vline&-1&1&1&1\\
\phi_3&\vline&1&-1&1&1\\
\phi_4&\vline&1&1&-1&-1\\
\phi_5&\vline&1&-1&1&-1\\
\phi_6&\vline&1&-1&-1&1
\end{bmatrix}\]

Note that $\Phi$ fails to do phase retrieval as it requires seven vectors in $\RR^4$ to do phase retrieval in $\RR^4$.
Given $x=(a_1,a_2,a_3,a_4),\ y=(b_1,b_2,b_3,b_4)$ we assume
\begin{equation}\label{E}
|\langle x,\phi_i\rangle|^2=|\langle y,\phi_i\rangle|^2, \text{ for all } i=1,2,3,4,5,6.
\end{equation}
\vskip12pt

\noindent {\bf Step 1:} The following is the coefficient matrix obtained after expanding $\langle x, \phi_i\rangle|^2$ for $i=1, 2, \ldots, 6$.
\[\frac{1}{2}\begin{bmatrix}
&\vline&a_1a_2&a_1a_3&a_1a_4&a_2a_3&a_2a_4&a_3a_4&\sum_{i=1}^4a_i^2\\
E_1&\vline&1&1&-1&1&-1&-1&\frac{1}{2}\\
E_2&\vline&-1&-1&-1&1&1&1&\frac{1}{2}\\\
E_3&\vline&-1&1&1&-1&-1&1&\frac{1}{2}\\\
E_4&\vline&1&-1&-1&-1&-1&1&\frac{1}{2}\\\
E_5&\vline&-1&1&-1&-1&1&-1&\frac{1}{2}\\\
E_6&\vline&-1&-1&1&1&-1&-1&\frac{1}{2}\
\end{bmatrix}\]

\noindent {\bf Step 2:} Consider the following row operations, the last column becomes all zeroes so we drop it and we get:

\[\begin{bmatrix}
F_1=\frac{1}{2}(E_1-E_4)&\vline&0&1&0&1&0&-1\\
F_2=\frac{1}{2}(E_2-E_5)&\vline&0&-1&0&1&0&1\\
F_3=\frac{1}{2}(E_3-E_6)&\vline&0&1&0&-1&0&1\\
A_1=\frac{1}{2}(F_1+F_2)&\vline&0&0&0&1&0&0\\
A_2=\frac{1}{2}(F_1+F_3)&\vline&0&1&0&0&0&0\\
A_3=\frac{1}{2}(F_2+F_3)&\vline&0&0&0&0&0&1\\
\end{bmatrix}\]

\noindent {\bf Step 3:} Subtracting out $A_1, A_2$ and $A_3$ from $E_1, E_2, E_3$ and $E_4$, we get:

\[\begin{bmatrix}
E_1'=&\vline&1&0&-1&0&-1&0\\
E_2'=&\vline&-1&0&-1&0&1&0\\
E_3'=&\vline&-1&0&1&0&-1&0\\
E_4'=&\vline&1&0&-1&0&-1&0\\
\end{bmatrix}\]

\noindent {\bf Step 4:}  We will show that $a_ia_j=b_ib_j$ for all $i\not= j$.

Performing the given operations we get:

\[\begin{bmatrix}
D_1=\frac{-1}{2}(E_2'+E_3')&\vline&1&0&0&0&0&0\\
A_2&\vline&0&1&0&0&0&0\\
D_2=\frac{-1}{2}(E_1'+E_2')&\vline&0&0&1&0&0&0\\
A_1&\vline&0&0&0&1&0&0\\
D_3=\frac{-1}{2}(E_3'+E_4')&\vline&0&0&0&0&1&0\\
A_3 &\vline&0&0&0&0&0&1\\
\end{bmatrix}\]

Doing the same operations with $y=(b_1,b_2,b_3, b_4)$ we get:
\[ a_ia_j = b_ib_j,\mbox{ for all } 1 \le i \not= j \le 4.\]

\begin{remark} It should be noted that weak phase retrieval does not imply norm retrieval. We may use the previous example to illustrate this. Let $\Phi=\{\phi_i\}_{i=1}^6$ be as in Example~$\ref{Ex:R4}$. Suppose $\Phi$ does norm retrieval. Since there are only 6 vectors $\Phi$ fails the complement property. Now, take $x=(1,1,-1,1)\perp \{\phi_1, \phi_2, \phi_3\}$ and $y=(1,1,1,1)\perp \{\phi_4, \phi_5, \phi_6\}$. Then, we have $|\langle x+y,\phi_i|=|\langle x-y, \phi_i\rangle|$ for all $i=1,2, \ldots 6.$ From the definition~\ref{D:norm retrieval}, this implies $\norm{x+y}=\norm{x-y}$.  Since
$\|x\|=\|y\|$, this implies that $x \perp y$, which is a contradiction.
\end{remark}


\end{document}